\documentclass[12pt,leqno,a4paper,draft]{amsart}
\usepackage{amsmath, amsthm, verbatim, amsfonts, amssymb, color}
\usepackage{mathrsfs}
\usepackage{dsfont}
\usepackage[a4paper,margin=1in]{geometry}

\theoremstyle{plain}

\newtheorem{corollary}{Corollary}
\newtheorem{lemma}{Lemma}

\newtheorem*{theorem*}{Theorem}
\newtheorem*{corollary*}{Corollary}
\newtheorem*{lemma*}{Lemma}
\newtheorem*{proposition*}{Proposition}

\theoremstyle{definition}

\newtheorem*{remark*}{Remark}

\numberwithin{equation}{section}
\renewcommand\labelenumi{\textup{\alph{enumi})}}
\renewcommand\theenumi\labelenumi

\makeatletter
\def\@makefnmark{\hbox{(\@textsuperscript{\normalfont\@thefnmark})}}
\makeatother
\DeclareFontFamily{U}{mathx}{\hyphenchar\font45}
\DeclareFontShape{U}{mathx}{m}{n}{
      <5> <6> <7> <8> <9> <10>
      <10.95> <12> <14.4> <17.28> <20.74> <24.88>
      mathx10
      }{}
\DeclareSymbolFont{mathx}{U}{mathx}{m}{n}
\DeclareFontSubstitution{U}{mathx}{m}{n}
\DeclareMathAccent{\widecheck}{0}{mathx}{"71}
\DeclareMathAccent{\wideparen}{0}{mathx}{"75}

\newcommand\nat{\mathds{N}}

\newcommand\real{{\mathds{R}}}
\newcommand\comp{{\mathds{C}}}
\newcommand\HH{{\mathds{H}}}
\newcommand\I{\mathds{1}}

\newcommand\id{\mathop{\mathrm{id}}}
\newcommand\supp{\mathop{\mathrm{supp}}}
\newcommand\RE{\mathop{\mathrm{Re}}}
\newcommand\rn{{\mathds{R}^n}}

\newcommand{\Dcal}{\mathcal{D}}
\newcommand{\Lcal}{\mathcal{L}}
\newcommand{\Scal}{\mathcal{S}}
\newcommand{\boxperp}{[\perp]}

\newcommand{\scalp}[2]{\langle #1,\,#2\rangle}

\usepackage{xcolor}
\newcommand\normal{\color{black}}
\newcommand\rene{\color{red}}

\usepackage{marginnote}
\begin{document}
\title[The Liouville property for L\'evy generators]{\bfseries On the Liouville property for non-local L\'evy generators}


\author[V.~Knopova]{Victoriya Knopova}

\author[R.L.~Schilling]{Ren\'e L.\ Schilling}
\address{TU Dresden\\ Fakult\"{a}t Mathematik\\ Institut f\"{u}r Mathematische Stochastik\\ 01062 Dresden, Germany}
\email{victoria.knopova@tu-dresden.de, rene.schilling@tu-dresden.de}
\thanks{\emph{Acknowledgement}. We thank Moritz Kassmann for drawing our attention to the preprint \cite{ali-et-al} of Alibaud, del Teso, Endal and Jakobsen. We are grateful to Espen Jakobsen who sent us the latest version of \cite{ali-et-al} and whose comments were most helpful. Bj\"{o}rn B\"{o}ttcher, Wojciech Cygan, Franziska K\"{u}hn, Niels Jacob and Zolt\'an Sasv\'ari read various earlier versions, pointed out mistakes and made valuable suggestions -- a big thank you, too.}

\begin{abstract}
    We prove a necessary and sufficient condition for the Liouville property of the infinitesimal generator of a L\'evy process and subordinate L\'evy processes. Combining our criterion with the necessary and sufficient condition obtained by Alibaud \emph{et al.}, we obtain a characterization of  (orthogonal subogroup of the)  the set of zeros of the characteristic exponent of the L\'evy process.
\end{abstract}

\maketitle

A $C^2$-function $f:\rn\to\real$ is called \textbf{harmonic}, if $\Delta f = 0$ for the Laplace operator $\Delta$. The classical Liouville theorem states that any bounded harmonic function is constant. Often it is helpful to understand $\Delta f$ as a Schwartz distribution in $\Dcal'(\rn)$ and to re-formulate the Liouville problem in the following way: The operator $\Delta$ enjoys the \textbf{Liouville property} if
\begin{gather}\label{liou}
    f\in L^\infty(\rn) \;\;\text{and}\;\; \forall \phi\in C_c^\infty(\rn)\::\:\scalp{\Delta f}{\phi} := \scalp{f}{\Delta \phi} = 0 \implies f\equiv\textup{const.}
\end{gather}
holds; $\scalp{\cdot}{\cdot}$ denotes the (real) dual pairing used in the theory of distributions. An excellent account on the history and the importance of the Liouville property can be found in the paper \cite{ali-et-al} by Alibaud \emph{et al.}

Like \cite{ali-et-al} we are interested in the analogue of \eqref{liou} for a class of non-local operators with constant `coefficients'. Recall that $\frac 12\Delta$ is the infinitesimal generator of Brownian motion. This is a diffusion process with stationary and independent increments and continuous paths. If we give up the continuity of the sample paths, and consider stochastic processes with independent, stationary increments and right-continuous paths with finite left-hand limits, we get the family of \textbf{L\'evy processes}, cf.\ \cite{barca}. It is well known, see \cite{jac-1,barca}, that the infinitesimal generator $\Lcal_\psi$ of a L\'evy process is a \textbf{pseudo-differential operator}
\begin{gather}\label{pseudo}
    \widehat{\Lcal_\psi u}(\xi) = -\psi(\xi)\widehat u(\xi),\quad u\in \Scal(\rn)
\end{gather}
where $\widehat u(\xi) = (2\pi)^{-n}\int_\rn e^{-ix\cdot\xi} u(x)\,dx$ is the Fourier transform and $\Scal(\rn)$ is the Schwartz space of rapidly decreasing smooth functions; the inverse Fourier transform is denoted by $\widecheck u$.  The \textbf{symbol} $\psi:\rn\to\comp$ is a \textbf{continuous and negative definite} function which is uniquely characterized by its \textbf{L\'evy--Khintchine representation}
\begin{gather}\label{lkf}
    \psi(\xi)
    = -ib\cdot\xi + \frac 12 Q\xi\cdot\xi + \int_{\rn\setminus\{0\}}\left(1-e^{i x\cdot\xi}-ix\cdot\xi\I_{(0,1)}(|x|)\right)\nu(dx);
\end{gather}
the `coefficients' $b\in\rn$, $Q\in\real^{n\times n}$ (a positive semidefinite matrix) and $\nu$ (a Radon measure on $\rn\setminus\{0\}$ such that $\int_{\rn\setminus\{0\}} \min\{|x|^2,1\}\,\nu(dx)<\infty$) uniquely describe $\psi$.
From \eqref{pseudo} we can see that the $L^2$-adjoint $\Lcal_\psi^*$ is again a L\'evy generator whose symbol $\overline\psi$ is the complex conjugate of $\psi$: Let $u,\phi\in\Scal(\rn)$. By Plancherel's theorem
\begin{gather}\label{plancherel}
    \scalp{\Lcal_\psi u}{\phi}
    =\scalp{\widehat{\Lcal_\psi u}}{\widecheck \phi}
    =\scalp{-\psi\widehat{u}}{\widecheck \phi}
    =\scalp{\widehat{u}}{-\psi\widecheck \phi}
    =\scalp{\widehat{u}}{\widecheck{\Lcal_{\overline\psi} \phi}}
    =\scalp{u}{\Lcal_{\overline\psi} \phi}.
\end{gather}

If we combine \eqref{pseudo} and \eqref{lkf} we get a further representation for $\Lcal_\psi$ on $\Scal(\rn)$ 
\begin{gather}\label{integro}\begin{aligned}
    \Lcal_\psi u(x)
    &= b\cdot\nabla u(x) + \frac 12\nabla \cdot Q\nabla u(x) \\
    &\qquad\mbox{} + \int_{\rn\setminus\{0\}} \left(u(x+y)-u(x)-y\cdot\nabla u(x)\I_{(0,1)}(|y|)\right)\nu(dy).
\end{aligned}\end{gather}
Using Taylor's formula it is not hard to see that $\|\Lcal_\psi u\|_{L^p} \leq c\sum_{|\alpha|\leq 2} \|\partial^\alpha u\|_{L^p}$ for all $1\leq p\leq \infty$ and $u\in \Scal(\rn)$, see e.g.~\cite[Lemma~3.4]{ieot}. 
If we use $p=1$ and $\overline\psi$ instead of $\psi$, we conclude that
$\|\Lcal_\psi^*\phi\|_{L^1} = \|\Lcal_{\overline\psi}\phi\|_{L^1} \leq c\sum_{|\alpha|\leq 2} \|\partial^\alpha \phi\|_{L^1} \leq c_{K}\sum_{|\alpha|\leq 2} \|\partial^\alpha \phi\|_{L^\infty}$ for all $\phi\in C_c^\infty(K)$ and all compact sets $K\subset\rn$. This is already the proof of the following lemma\footnote{The lemma is also a direct consequence of the positive maximum principle: If $\phi\in \Scal(\rn)$ attains a positive maximum at $x_0$, then $\Lcal_\psi\phi(x_0)\leq 0$; see \cite[Theorem 6.2]{barca} and \cite[Theorem 2.21]{LM3}.}.
\begin{lemma}\label{lemma-1}
    Let $\Lcal_\psi$ be the generator of a L\'evy process. For every $f\in L^\infty(\rn)$
    \begin{gather}\label{distri}
        \phi \mapsto \scalp{\Lcal_\psi f}{\phi} := \scalp{f}{\Lcal_\psi^*\phi} = \int_\rn f(x) \Lcal_{\overline\psi}\phi(x)\,dx,
        \quad\phi\in C_c^\infty(\rn)
    \end{gather}
    is a Schwartz distribution $\Lcal_\psi f\in\Dcal'(\rn)$ of order $\leq 2$.
\end{lemma}
The functional \eqref{distri} has a natural extension to all $\phi\in\Scal'(\rn)$ such that $\Lcal_{\overline\psi} \phi \in L^1(\rn)$. The latter condition means that $\phi\in H^{\psi,2}_1(\rn)$ where $H^{\psi,2}_1(\rn)$ is an anisotropic Bessel potential space, cf.\ \cite{fjs-2}. Since $\Scal(\rn)$ is dense in $H^{\psi,2}_1(\rn)$, \cite[Theorem 2.1.13]{fjs-2}, the extension is an extension by continuity.

We want to characterize those L\'evy generators which satisfy the Liouville property. For this, we need the following results on the zero set of a continuous negative definite function.
\begin{lemma}\label{lemma-2}
    Let $\psi:\rn\to\comp$ be the symbol of the generator of a L\'evy process. The following estimates hold
    \begin{gather}\label{sub-add}
        \sqrt{|\psi(\xi+\eta)|} \leq \sqrt{|\psi(\xi)|} + \sqrt{|\psi(\eta)|},\quad\xi,\eta\in\rn;\\
    \label{period}
        |\psi(\xi)+\overline{\psi(\eta)}-\psi(\xi-\eta)| \leq 4|\psi(\xi)| |\psi(\eta)|,\quad\xi,\eta\in\rn;
    \end{gather}
    in particular, the zero-set $\{\psi=0\} = \{\xi\in\rn : \psi(\xi) = 0\}$ is a closed subgroup of $(\rn,+)$ which is either discrete or of the form $G\oplus E$ for some subspace $E\subset\rn$ and a closed discrete subgroup \textup{(}a relative lattice\textup{)} $G\subset E^\bot$.
    If $\{\psi=0\}\supsetneq\{0\}$ is not trivial, $\psi$ must be periodic, hence bounded.
\end{lemma}
\begin{proof}
    The inequalities \eqref{sub-add}, \eqref{period} are well-known, see Berg and Forst \cite[Proposition 7.15]{ber-for}, Jacob \cite[Lemma 3.621]{jac-1} or \cite[Theorem 6.2]{barca} for various proofs. Since $\psi(-\eta)=\overline{\psi(\eta)}$, the inequality \eqref{sub-add} and the continuity of $\psi$ show that $\{\psi=0\}$ is a closed subgroup of $\rn$. The decomposition follows from a structure result on closed subgroups of $\rn$, see
    Bourbaki \cite[Chapter VII, Theorem 1.2.2]{bourbaki}.  If $\eta\in\{\psi=0\}$ and $\eta\neq 0$, the estimate \eqref{period} shows that $\psi$ is periodic with period $\eta$. Since $\psi$ is continuous, it is bounded.
\end{proof}

\begin{lemma}\label{lemma-3}
    Let $\psi:\rn\to\comp$ be the symbol of the generator of a L\'evy process with coefficients $(b,Q,\nu)$, and denote by $\psi_n$ the symbol with coefficients $(b,Q,\I_{B_n(0)}\nu)$. Then $\psi_n\in C^\infty(\rn)$, $\lim_{n\to\infty}\psi_n = \psi$ uniformly and the zero-sets satisfy
    \begin{gather}\label{zero-set}
        \{\psi = 0\} = \bigcap_{n\in\nat}\{\psi_n = 0\}.
    \end{gather}
\end{lemma}
\begin{proof}
    Since the L\'evy measure of $\psi_n$ has bounded support, all moments $\int_{1<|x|< n} |x|^k\,\nu(dx)$, $k\in\nat_0$, are finite, and a routine application of parameter-dependent integrals shows that $\psi_n$ is of class $C^\infty$, see also Jacob \cite[Theorem 3.7.13]{jac-1} or \cite[Lemma 2.18]{LM3}. The difference of $\psi-\psi_n$ can be estimated by
    \begin{gather*}
        |\psi(\xi) - \psi_n(\xi)| \leq \int_{|x|\geq 1} |1-e^{ix\cdot\xi}|\,\nu(dx) \leq 2\nu(B_n^c(0))
    \end{gather*}
    which converges to $0$ uniformly in $\xi$ as $n\to\infty$.

    In order to show \eqref{zero-set} we consider first the real case.
    Since $\RE\psi_n \leq \RE\psi$, we have $\{\RE\psi = 0\}\subset \{\RE\psi_n = 0\}$ and $\{\RE\psi = 0\}\subset \bigcap_{n\in\nat}\{\RE\psi_n = 0\}$. The converse inclusion follows from the pointwise convergence $\psi_n(\xi)\to\psi(\xi)$.

    In general, $\psi$ is bounded if, and only if, $\nu$ is a finite measure, and $\psi$ is of the form
    \begin{gather*}
        \psi(\xi) = \int_{\rn\setminus\{0\}}\left(1-e^{ix\cdot\xi}\right)\nu(dx);
    \end{gather*}
    this follows easily from \eqref{lkf}, see also Berg and Forst \cite[Proposition 7.13]{ber-for}. Since $\nu(B_1^c(0))<\infty$,  we see that $\psi$ and $\psi_n$ are at the same time bounded or unbounded.

    If $\psi$ and $\psi_n$ are unbounded, we have $\{\psi=0\}=\{0\}$ and $\{\psi_n=0\}=\{0\}$; this follows from Lemma~\ref{lemma-2}.

    Assume now that $\psi$ and $\psi_n$ are bounded. There are two cases.

    \noindent
    \emph{Case 1:} There is a sequence $n(k)\to\infty$ such that $\{\RE\psi_{n(k)} = 0\}\supsetneq\{0\}$.  This means that the L\'evy measure of $\psi_{n(k)}$ is purely atomic; by construction, it follows that  $\nu$ is on each ball $B_{n(k)}(0)$ of the form
    \begin{gather*}
        \I_{B_{n(k)}}(x)\nu(dx) = \sum_{j\,:\, |b_j| < n(k)} a_j\delta_{b_j}(dx)
        \quad\text{and}\quad
        \nu(dx) = \sum_{j\in\nat} a_j\delta_{b_j}(dx).
    \end{gather*}
    Observe that
    \begin{gather*}
        1-e^{i\xi\cdot b_j}=0
        \iff 1-\cos \xi\cdot b_j = 0
        \iff \sin\xi\cdot b_j=0.
    \end{gather*}
    Thus, $\{\RE\psi_n = 0\}=\{\psi_n = 0\}$ and $\{\RE\psi = 0\}=\{\psi = 0\}$. Since \eqref{zero-set} is true for $\RE\psi_n$ and $\RE\psi$, we get
    \begin{gather*}
        \bigcap_{n\in\nat} \{\psi_n = 0\}
        =\bigcap_{n\in\nat} \{\RE\psi_n = 0\}
        = \{\RE\psi = 0\}
        = \{\psi = 0\}.
    \end{gather*}

    \noindent
    \emph{Case 2:} There is some $n_0$ such that $\{\RE\psi_n=0\}=\{0\}$ for all $n\geq n_0$. Since \eqref{zero-set} is true for $\RE\psi_n$ and $\RE\psi$, we have again
    \begin{gather*}
        \{0\}
        = \bigcap_{n\in\nat} \{\RE\psi_n = 0\}
        = \{\RE\psi = 0\}
        \supset\{\psi = 0\}\supset\{0\},
    \end{gather*}
    and \eqref{zero-set} follows from this and the trivial inclusions $\{\RE\psi_n=0\}\supset\{\psi_n=0\}\supset\{0\}$.
\end{proof}

We can now prove our main theorem.

\begin{theorem*}
    Let $\Lcal_\psi$ be the generator of a L\'evy process. The operator $\Lcal_\psi$ has the Liouville property, i.e.\
    \begin{gather}\label{liouville-levy}
        f\in L^\infty(\rn)\;\;\text{and}\;\;\forall\phi\in C_c^\infty(\rn)\::\:\scalp{\Lcal_\psi f}{\phi} = 0
        \implies f\equiv\textup{const.}
    \end{gather}
    if, and only if, the zero-set satisfies $\{\psi=0\}=\{0\}$.
\end{theorem*}
\begin{proof}
    Let $f\in L^\infty(\rn)$ and $\Lcal_\psi f = 0$ in $\Dcal'(\rn)$. In view of Lemma~\ref{lemma-2} we can assume that $\{\psi=0\}$ is a discrete group, otherwise we would have a truly lower-dimensional problem.

    Assume first that $\psi$ is smooth and that all derivatives grow at most polynomially. Since such $\psi$ are multipliers in $\Scal(\rn)$, we see immediately that in $\Scal'(\rn)$
    \begin{gather*}
        0
        = \scalp{\Lcal_\psi f}{\phi}
        = \scalp{f}{\Lcal_{\overline\psi} \phi}
        = \scalp{\widecheck f}{\widehat{\Lcal_{\overline\psi} \phi}}
        = \scalp{\widecheck f}{\overline\psi \widehat{\phi}}
    \end{gather*}
    holds for all $\phi\in\Scal(\rn)$ and all $\widehat\phi\in\Scal(\rn)$ (as the Fourier transform is bijective on $\Scal(\rn)$). We conclude that $\supp\widecheck f\subset \{\overline\psi=0\}=\{\psi=0\}$.

    If $\psi$ is a general L\'evy symbol, we use the smooth approximation $\psi_n$ from Lemma~\ref{lemma-3}. It is known, see Jacob~\cite[Thorem~3.7.13]{jac-1} or \cite[Lemma 2.18]{LM3}, that $\psi_n$ is smooth and grows, together with all derivatives, at most quadratically. Note that for $\phi\in\Scal(\rn)$
    \begin{gather*}
        \widehat{\Lcal_{\overline\psi_n}\phi}
        = \overline\psi_n\widehat\phi
        = \overline\psi\,\frac{\overline\psi_n}{\overline\psi}\,\widehat\phi
        \in \Scal(\rn)
        \implies
        \Lcal_{\overline\psi_n}\phi = \Lcal_{\overline\psi}\left(\widecheck{\frac{\overline\psi_n}{\overline\psi}\,\widehat\phi}\right)
    \end{gather*}
    which means that $\Lcal_\psi f = 0$ implies $\Lcal_{\psi_n}f=0$ for all $f\in L^\infty(\rn)$. Thus, the previous arguments show $\supp f\subset\{\psi_n=0\}$ for all $n\in\nat$ and, again with Lemma~\ref{lemma-3}, we get $\supp f\subset \bigcap_{n\in\nat}\{\psi_n=0\} = \{\psi=0\}$.%

    We can now use a classical result on the structure of tempered distributions supported in single points, see e.g.\
    Tr\`eves \cite[Chapter 24]{treves},
    \begin{gather*}
        \smash[b]{\widecheck f = \sum_{g\in\{\psi=0\}} \sum_{|\alpha|\leq n(g)<\infty} c_{\alpha,g} \partial^\alpha\delta_g}
    \intertext{or, equivalently,} 
        f(x) = \sum_{g\in\{\psi=0\}} \sum_{|\alpha|\leq n(g)<\infty} (2\pi)^{-n}c_{\alpha,g} (ix)^\alpha e^{-ig\cdot x}.
    \end{gather*}
    If, in addition, $f$ is bounded, we have $f(x) = (2\pi)^{-n}\sum_{g\in\{\psi = 0\}} c_{0,g} e^{-ig\cdot x}$.

    If $\{\psi=0\}=\{0\}$, it is clear that the Liouville property holds. Conversely, if the Liouville property holds and if $\{\psi=0\}$ is not trivial, we can always shift a solution $\Lcal_\psi f=0$ in the following way: $f_h(x):= f(x)e^{-ih\cdot x}$, $h\in\{\psi=0\}$ and we get again $\Lcal_\psi f_h = 0$; this means that can always achieve that some $c_{0,g}\neq 0$, $g\neq 0$, and we have reached a contradiction to $f_h\equiv\text{const.}$ Thus, the above representation shows that the Liouville property holds if, and only if, $\{\psi = 0\}$ is trivial. If $\{\psi = 0\}\supsetneq\{0\}$, the solutions to $\Lcal_\psi f = 0$ are periodic with periodicity group given by the orthogonal subgroup (cf.\ \cite[Definition~2.8]{ber-for}) of the zero-set $\{\psi = 0\}^{\boxperp} := \{\xi\in\rn \mid \forall g\in\{\psi=0\} \,:\, e^{i\xi\cdot g}=1\}$.
\end{proof}

Using \textbf{Bochner's subordination}, cf.\ \cite[Chapter 13]{SSV}, we can give a further characterization of the Liouville property. If $(T_t)_{t\geq 0}$ is any strongly continuous contraction semigroup and $(\gamma_t)_{t\geq 0}$ a vaguely continuous convolution semigroup of probability measures on the half line $[0,\infty)$, then $T^g_t u := \int_0^\infty T_su\,\gamma_t(ds)$ (understood as a Bochner integral) is again a strongly continuous contraction semigroup. The superscript $g$ in $T^g_t$ denotes a \textbf{Bernstein function}. Bernstein functions uniquely characterize the convolution semigroup $(\gamma_t)_{t\geq 0}$ via the (one-sided) Laplace transform $\widetilde\gamma_t(\lambda) := \int_{[0,\infty)} e^{-\lambda s}\,\gamma_t(ds) = e^{-t g(\lambda)}$, and all Bernstein functions are of the form
\begin{gather}\label{bernstein}
    g(\lambda) = a\lambda + \int_{(0,\infty)} \left(1-e^{-\lambda s}\right)\pi(ds)
\end{gather}
where $a\geq 0$ and $\pi$ is a measure on $(0,\infty)$ satisfying $\int_0^\infty \min\{s,1\}\,\pi(ds)<\infty$, see \cite[Chapter 3]{SSV}. Of the many properties of Bernstein functions let us only note that any $g\not\equiv 0$ is strictly monotone. Typical examples of Bernstein functions are fractional powers $g(\lambda)=\lambda^\alpha$, $0<\alpha<1$, ($a=0$, $\pi(ds) = \frac{\alpha}{\Gamma(1-\alpha)} s^{-1-\alpha}\,ds$), logarithms $g(\lambda)=\log(1+\lambda)$ ($a=0$, $\pi(ds) = s^{-1}e^{-s}\,ds$), `resolvents' $g(\lambda) = \frac{\lambda}{\tau+\lambda}$, $\tau>0$, ($a=0$, $\pi(ds) = \tau e^{-s\tau}\,ds$) and `semigroups' $g(\lambda) = 1-e^{-\lambda t}$, $t>0$, ($a=0$, $\pi(ds) = \delta_t(ds)$).

If we extend $g$ to the complex right-half plane $\HH = \{\lambda+i\eta \mid \lambda\geq 0, \eta\in\real\}$, we see from \eqref{bernstein} that
\begin{gather}\label{ext}
    \RE g(\zeta) = a\lambda + \int_{(0,\infty)} \left(1-e^{-\lambda s}\cos(\eta s)\right)\pi(ds),\quad \zeta=\lambda+i\eta \in\HH.
\end{gather}
Note that $e^{-\lambda s} < 1$ if $\lambda, s > 0$. Thus, if $a>0$, $g(\zeta)=0$ if, and only if, $\zeta=0$. If $a=0$ and $\lambda>0$, the inequality
\begin{gather*}
    1-e^{-\lambda s}\cos(\eta s) > 1-|\cos(\eta s)| \geq 0,\quad \lambda, s>0
\end{gather*}
shows that $g(\zeta)=0$ can only happen if $\lambda=0$. In this case we also need that
\begin{gather*}
    \int_{(0,\infty)} \left(1-\cos(\eta s)\right)\pi(ds) = 0
\end{gather*}
which is only possible for $\eta\neq 0$ if $\supp\pi$ is discrete\footnote{If we compare \eqref{lkf} and \eqref{bernstein} for $\lambda = i\eta$, we see that $\eta\mapsto f(i\eta)$ is a continuous and negative definite function, and the exact structure of $\supp\pi$ is given in Corollary~\ref{cor-2} below.}.

If $\psi$ is the symbol of a L\'evy process, then $g\circ\psi$ is also the symbol of a L\'evy process and $-g(-\Lcal_\psi) = \Lcal_{g\circ\psi}$ on $\Scal(\rn)$; here, $-g(-\Lcal_\psi)$ is understood as a function of the operator $-\Lcal_\psi$ in virtually any reasonable functional calculus sense, see \cite[Chapter 13]{SSV}.

\begin{corollary}\label{cor-1}
    Let $\Lcal_\psi$ be the generator of a L\'evy process and $g$ be a Bernstein function given by \eqref{bernstein}. If $a\neq 0$ or if $\supp\pi$ is not discrete, then $\Lcal_\psi$ has the Liouville property if, and only if, $\Lcal_{g\circ\psi}$ has the Liouville property.

    In particular, $\Lcal_\psi$ has the Liouville property if, and only if, for some \textup{(}or all\textup{)} $\tau>0$ the resolvent $R_\tau := (\tau\id-\Lcal_\psi)^{-1}$ enjoys the following property:
    \begin{gather}\label{liou-res}
        f\in L^\infty(\rn)\;\;\text{and}\;\; \tau R_\tau f = f \implies f\equiv\textup{const.}
    \end{gather}

    If, in addition, $\Lcal_\psi$ is self-adjoint, i.e.\ if $\psi = \overline\psi$ is real-valued, then the Liouville property of $\Lcal_\psi$ is equivalent to the following property of the semigroup $P_t = e^{t\Lcal_\psi}$ generated by $\Lcal_\psi$: For some \textup{(}or all\textup{)} $t>0$ one has
    \begin{gather}\label{liou-semi}
        f\in L^\infty(\rn)\;\;\text{and}\;\; P_t f = f \implies f\equiv\textup{const.}
    \end{gather}
\end{corollary}
\begin{proof}
    The first assertion follows from the observation that $\{\zeta\in\HH \mid g(\zeta)=0\}=\{0\}$ implies $\{g\circ \psi = 0\}=\{\psi = 0\}$ and the main Theorem.

    For the second part use the Bernstein function $g(\lambda)=\frac{\lambda}{\tau+\lambda}$ and note that $\lambda=0$ is the only zero of $g$ in $\HH$; moreover $\widehat {R_\tau\phi} = \psi (\tau+\psi)^{-1}\widehat\phi$.

    Finally, if $\psi$ is real-valued, i.e.\ if $\Lcal_\psi$ is self-adjoint cf.\ \eqref{bernstein}, there is no need to extend $g$ to $\HH$. In this case, $\{\lambda \geq 0 \mid g(\lambda)=0\}=\{0\}$ is always trivial  (because of the strict monotonicity of $g$)  and we can use the previous argument for $g(\lambda) = 1-e^{-\lambda t}$ and the semigroup $P_t$: Note that $\widehat{P_t\phi} = e^{-t\psi}\widehat\phi$.
\end{proof}

The paper \cite{ali-et-al} by Alibaud \emph{et al.}\ contains another characterization of the Liouville property for L\'evy generators using completely different methods: It is based on the characteristic triplet $(b,Q,\nu)$ appearing in the L\'evy--Khintchine representation \eqref{lkf} of $\psi$. If we combine our Theorem with the result of \cite{ali-et-al}, we arrive at an interesting description of the zero-set of a negative definite function.

We need the following notation from \cite{ali-et-al}.
Let $\Sigma$ be the positive semidefinite square root of $Q$ and denote by $\sigma_1,\dots,\sigma_n$ the column vectors of $\Sigma$.
Let $G_\nu = G(\supp\nu)$ be the smallest additive subgroup of $\rn$ containing $\supp\nu$, $V_\nu = \{x\in\overline{G_\nu} \::\: tx\in\overline{G_\nu}\;\forall t\in\real\}$  ($\overline G$ stands for the closure of $G\subset\rn$),  $c_\nu = -\int_{\{|y|<1\}\setminus V_\nu} y\,\nu(dy)$ and $W_{\Sigma,b+c_\nu} = \mathrm{span}\{\sigma_1,\dots,\sigma_n, b+c_\nu\}$.
\begin{corollary}\label{cor-2}
    Let $\psi$ be the characteristic exponent of a L\'evy process. Denote by $\{\psi = 0\}^{\boxperp}$ the orthogonal subgroup $\{\xi\in\rn \mid \forall x\in\{\psi=0\} : e^{i\xi\cdot x}=1\}$ of the additive group $\{\psi=0\}\subset\rn$. Then the following equality holds
    \begin{gather}
        \{\psi = 0\}^{\boxperp}
        = \overline{G_\nu + W_{\Sigma,b+c_\nu}}.
    \end{gather}
\end{corollary}
There is a further characterization of $\{\psi=0\}^{\boxperp}$ which can be found in Berg and Forst \cite[Proposition 8.27]{ber-for}. Denote by $\mu_t$ the probability measure such that $\widehat \mu_t = e^{-t\psi}$, i.e.\ $(\mu_t)_{t\geq 0}$ is the family of transition probabilities of the L\'evy process with exponent $\psi$. Then $\{\psi=0\}^{\boxperp}$ is the smallest closed additive subgroup of $\rn$ which contains $\bigcup_{t>0} \supp\mu_t$.  This implies immediately the following result.
\begin{corollary}\label{cor-3}
    Assume that the L\'evy process with characteristic function $\psi$ has transition probabilities $(\mu_t)_{t\geq 0}$ such that for at least one $t_0>0$ the measure $\mu_{t_0}(dx) = p_{t_0}(x)\,dx$ has a strictly positive density, i.e.\ $p_{t_0}(x) > 0$ for all $x\in\rn$. Then the generator $\Lcal_\psi$ has the Liouville property \eqref{liouville-levy}.
\end{corollary}

\end{document}